\documentclass{article}
\usepackage[utf8]{inputenc}
\usepackage{amsmath, amssymb, amsthm,mathrsfs}
\usepackage{hyperref}
\usepackage[all]{xy}    
\usepackage{graphicx}   
\usepackage{authblk}    
\usepackage{url}    

\newcommand{\ZZ}{\mathbb Z}

\renewcommand{\AA}{\mathbb A}

\newcommand{\FF}{\mathbb F}

\newcommand{\bsym}{\boldsymbol}

\newtheorem{lemm}{Lemma}[section]
\newtheorem{thm}[lemm]{Theorem}
\newtheorem{prop}[lemm]{Proposition}
\newtheorem{coro}[lemm]{Corollary}

\newcommand{\ceil}[1]{\left\lceil #1\right\rceil}
\newcommand{\angles}[1]{\left\langle #1\right\rangle}

\newcommand{\SL}{\mathrm{SL}}
\newcommand{\GL}{\mathrm{GL}}
\newcommand{\Hom}{\mathrm{Hom}}

\title{The \textit{v}-function in the wild McKay correspondence is not determined by the ramification filtration}
\author{Takahiro Yamamoto\thanks{Affiliation: Osaka City University Advanced Mathematical Institute (3-3-138, Sugimoto, Sumiyoshi-Ku, Osaka city, Osaka-fu),
email:mass.11235813@gmail.com}}
\date{}

\begin{document}

\maketitle
\begin{abstract}
The \(v\)-function appears in the wild McKay correspondence and it is an invariant of a Galois extension over a local field which measures about the ramification of extension. The ramification filtration is also such invariant. Yasuda presented the problem whether the \(v\)-function is determined by ramification filtration. In this paper, we give a formula for the \(v\)-function for 2-dimensional representations of the bicyclic group of order \(p^2\) with \(p\) the characteristic of our base field. Using this formula, we construct an example of \(v\)-function not determined by the ramification filtration.
\end{abstract}

\tableofcontents

\section{Introduction}
The \(v\)-function appears in the wild McKay correspondence for quotient varieties \(\AA^d/G\) associated to linear action as a weighting function (see \cite{WY}, \cite{Yas3}). This is a function on the moduli space of \(G\)-\'etale \(K\)-algebras for some local field \(K\) and it is an invariant to study ramification of the algebras. The value of \(v\)-function coincides with the valuation of generator of \textit{the module resolvent} defined by Fr\"ohlich \cite{Fro} (see also \cite[Rem 9.4]{TY}).

As another invariant to study ramification of extensions of local fields, there is \textit{the ramification filtration}, which is a sequence \((G_i)_{i\geq -1}\) of subgroups of Galois group. Yasuda presented the problem whether the \(v\)-function is determined by the ramification filtration or not in \cite[Problem 5.1]{Yas}. The \(v\)-function has been computed in some cases. For the representation of \(p\)-cyclic group over the Laurent power series field \(k((t))\)) of characteristic \(p\), a formula of the \(v\)-function computed from the ramification jump, that is the index satisfying \(G_i\ne G_{i+1}\), in \cite{Yas2}. Similarly, for the representation of \(p^n\)-cyclic group, a formula of the \(v\)-function is determined by the ramification jumps in \cite[Theorem 3.11]{TanYas}. For any permutation action, the \(v\)-function is the same as the Artin conductor, which is determined by the ramification filtration \cite[Theorem 4.8]{WY}.

In this paper, we give a formula of the \(v\)-function for subgroup \(G\subset\SL_2(k)\) isomorphic to \((\ZZ/p\ZZ)^2\).
\begin{thm}[Main theorem, Theorem \ref{thm:main}]
Let \(k\) be an algebraically closed field of characteristic \(p>0\). Let \(K\) be the Laurent power series field \(k((t))\) and let \(\mathcal O_K\) be the valuation ring of \(K\). Let \(v_K\) be the valuation of \(K\).

For a subgroup \(G\subset\SL_2(k)\) generated by
\[\sigma=\begin{bmatrix}
1&1\\
0&1
\end{bmatrix},\ \tau=\begin{bmatrix}
1&a\\
0&1
\end{bmatrix}\quad(a\in k\backslash\FF_p),\]
we consider the induced \(G\)-representation \(V=\AA_{\mathcal O_K}^2\). For a \(G\)-Galois extension \(L/K\), let \(g_1,g_2\) be the corresponding elements of \(L^\sigma, L^\tau\) by the Artin-Schreier theory. Then
\[
\bsym v_{V}(g_1,g_2)=\ceil{-\frac{\min\{v_K(g_1),pv_K(a^pg_1+g_2)\}}{p^2}}.
\]
\end{thm}
We prove this theorem by direct computation. The \(v\)-function for linear action is computed from \textit{the tuning module}, which is the set of \(G\)-equivariant \(\mathcal O_K\)-linear homomorphisms from the linear part of the coordinate ring of \(V\) to \(\mathcal O_L\). First, we choose a \(K\)-basis of \(L\). In this situation, homomorphisms in the \textit{tuning module} is characterized by the value of \(x_2\), which is second coordinate of \(V\). We determine the sets of the value of \(x_2\) with respect to elements of the tuning module. From this computation, we get an \(\mathcal O_K\)-basis of the tuning module. By the defining formula of \(v\)-function using determinant, we get the formula in the above theorem.

From this formula, we get a counterexample for Yasuda's problem.
\begin{coro}[Corollary \ref{cor:main}]
We keep the notation of the main theorem. Then the value of \(\bsym v_V\) at a \(G\)-extension \(L/K\) is not determined by the ramification filtration of \(G\) associated to \(L/K\).
\end{coro}

We prove this by applying the main theorem to a special sequence of pairs \(g_1,g_2\) parameterized by \(c\in k-\FF_p\) such that every \(L_{g_1,g_2}\) has the same ramification filtration. We can see that if \(c=-a^2\) then the value of \(v\)-function at \(L_{g_1,g_2}\) is different from the one of the other case.

This paper is organized as follows. In section three, we recall the lower and upper ramification filtrations and basic facts about them. In section four, we give the definition of the \(v\)-function. In section five, we show the main theorem and its corollary.

\subsection*{Acknowledgements}
I am grateful to Takehiko Yasuda for teaching me and for his valuable comments. Without his help, this paper would not have been possible. This work was partially supported by JSPS KAKENHI Grant Number JP18H01112.
\section{Notation and Convention}
Throughout the paper, \(k\) denotes an algebraically closed field of characteristic \(p>0\) and \(K\) denotes the Laurent power series field \(k((t))\) and \(\mathcal O_K\) denotes the valuation ring of \(K\) with respect to the normalized valuation \(v_K\). 

For a finite group \(G\), a \(G\)\textit{-\'etale} \(K\)-\textit{algebra} \(A\) means an \'etale \(K\)-algebra endowed with \(G\)-action satisfying \(A^G\cong K\) and \(\dim_KA=\sharp G\).

\section{Ramification filtration}
The \textit{ramification filtration} is an invariant used to study the ramification of Galois extension of local field. For details, see \cite{Ser}. Let \(L/K\) be a Galois extension. We denote the Galois group of \(L/K\) by \(G\). Then \(G\) acts on the integral closure \(\mathcal O_L\) of \(\mathcal O_K\) in \(L\).

\textit{The i-th lower ramification group} \(G_i\) of \(G\) is defined by
\[
G_i=\{\sigma\in G\mid v_L(\pi\sigma-\pi)\geq i+1\}
\]
for \(i\geq -1\) where \(\pi\) is an uniformizer of \(\mathcal O_L\). The lower ramification groups give a descending sequence of subgroups of \(G\):
\[G=G_0\supset G_1\supset G_2\supset\cdots\supset G_i\supset G_{i+1}\supset\cdots\]
The sequence is called \text{the lower ramification filtration}.

In general, the lower ramification filtration is not compatible with quotient of group. To fix this problem, we consider \text{the upper ramification group} which is defined as follows. For a real number \(t\geq -1\), we define \(G_t=G_{\ceil t}\). We define a function \(\varphi\) by
\[
\varphi(x)=\int_0^x\frac{dt}{(G_0:G_t)}
\]
where \((G_0:G_t)\) is index of \(G_0/G_t\). Note that \((G_0:G_t)\) means the inverse of the index of \(G_0/G_{-1}\) for \(-1\leq t<0\) as convention. Then \(\varphi\) has its inverse function \(\psi\). We define
\[
G^v=G_{\psi(v)},
\]
which called \textit{the }\(v\)\textit{-th upper ramification group}. As \(\varphi\) is monotonically increasing, the upper ramification groups also gives a descending sequence of subgroups of \(G\):
\[G=G^0\supset G^1\supset G^2\supset\cdots\supset G^i\supset G^{i+1}\supset\cdots\]
This sequence is called \textit{the upper ramification filtration}.
\begin{prop}[Proposition 14, page 74, \cite{Ser}]\label{prop:upper}
The upper ramification filtration is compatible with quotients of the group. Namely, for a normal subgroup \(H\subset G\), we get
\[(G/H)^v=G^vH/H\]
for all \(v\geq -1\).
\end{prop}

\section{The \textit{\textbf{v}}-function}
In this paper, we define the \(v\)-function associated to a linear action. For more detail, see \cite{WY} or \cite{Yas4}. The \(v\)-function is a function from the set of \(G\)-\'etale \(K\)-algebras to \(\frac 1{\sharp G}\ZZ\). A \(G\)-\textit{\'etale} \(K\)-\textit{algebra} is a \(K\)-algebra endowed with \(G\)-action satisfying the stable subalgebra of \(G\) is isomorphic to \(K\).

Let \(V\) be an \(n\)-dimensional representation of a finite group \(G\subset\GL_n(k)\) over \(\mathcal O_K\). Note that the following definition can be applied for representaions of subsets of \(\GL_n(\mathcal O_K)\). We denote the coordinate ring of \(V\) by \(\mathcal O_K[x_1,\ldots,x_n]\). Then the linear part \(M=\sum_{i=1}^n\mathcal O_Kx_i\) of \(\mathcal O_K[x_1,\ldots,x_n]\) is the free \(\mathcal O_K\)-module of degree \(n\). For a \(G\)-\'etal \(K\)-algebra \(A\), We define its \textit{tuning module} \(\Xi_{A/K}^V\) by the set of all \(G\)-equivaliant \(\mathcal O_K\)-linear homomorphisms from \(M\) to \(\mathcal O_A\) where \(\mathcal O_A\) is the integral closure of \(\mathcal O_K\) in \(A\):
\[
\Xi_{A/K}^V=\Hom_{\mathcal O_K}^G(M,\mathcal O_A)
\]
Note that \(\Xi_{A/K}^V\) is also a free \(\mathcal O_K\)-module of degree \(n\). 
 The \(G\)-\'etale \(K\)-algebra \(A\) is written as
\[
A=L_1\oplus L_2\oplus\cdots\oplus L_m
\]
where \(L_i\) are copies of a Galois extension \(L/K\). The \(v\)-function \(\bsym v_V\) is defined by
\[
\bsym v_V(A)=\frac{1}{\sharp G}\mathrm{length}_{\mathcal O_K}\frac{\Hom_{\mathcal O_K}(M,\mathcal O_A)}{\mathcal O_A\cdot\Xi_{A/K}^V}
\]
in \cite[Definition 5.4]{Yas4}. Let \((\varphi_i)_{i=1}^n\subset\Xi_{A/K}^V\) be a \(\mathcal O_K\)-basis of \(\Xi_{A/K}^V\). Then the \(v_V\) is computed by
\begin{align}\label{form}
    \bsym v_V(A)=\frac 1{\sharp G}v_L(\det(\varphi_i(x_j))).
\end{align}

\section{Computation of \textit{\textbf v}-function}
We consider a subgroup \(G\) of \(\SL _2(k)\) generated by
\[\sigma=\begin{bmatrix}
1&1\\
0&1
\end{bmatrix},\ \tau=\begin{bmatrix}
1&a\\
0&1
\end{bmatrix}\]
where \(a\in k\backslash\FF_p\).

We denote the subset
\[
\bigoplus_{p\nmid j,\ j>0}kt^{-j}\subset K
\]
by \(J\). Note that \(J\) is a set of representatives of \(K/\mathcal P(K)\). We put
\[
J^{(2)}=\{(g_1,g_2)\in J^2\mid g_1\ne 0,\ g_2\not\in\FF_pg_1\}.
\]
Let \(L/K\) be a \(G\)-extension. Since the fixed fields \(L^\sigma,\ L^\tau\) are \(p\)-cyclic extensions, there exist uniquely \(\alpha,\ \beta\in L\) such that the following conditions holds.
\begin{itemize}
\item \(L^\sigma=K[\alpha]\) and \(L^\tau=K[\beta]\).
\item \(\alpha^p-\alpha,\ \beta^p-\beta\in J\).
\item \(\alpha\cdot\tau=\alpha+1,\ \beta\cdot\sigma=\beta+1\).
\end{itemize}
We denote \(\alpha^p-\alpha,\ \beta^p-\beta\) by \(g_1,\ g_2\), respectively. Then \((g_1,g_2)\in J^2\). In this case, we denote \(L\) by \(L_{g_1,g_2}\).

\begin{prop}\label{prop:paramExt}
We have the one-to-one correspondence:
\[\begin{array}{ccc}
\{G\text{-extensions }L/K\}&\leftrightarrow&J^{(2)}\\
 L_{g_1,g_2}&\mapsto&(g_1,g_2)
\end{array}\]
\end{prop}
\begin{proof}
We show that \((g_1,g_2)\in J^{(2)}\), which equivalent to that \(g_1\ne 0,\ g_2\not\in\FF_p g_1\), for a \(G\)-extension \(L_{g_1,g_2}\). If \(g_1=0\), then \(L^\sigma=K\), that contradicts the fact that \(L/K\) is \(G\)-extension. Thus we get \(g_1\ne 0\). If \(g_2\in\FF_pg_1\), then there exists \(0\ne s\in\FF_p\) such that \(g_2=sg_1\). Since \(s\in\FF_p\), we have \(s^p=s\). Hence
\[
(s\alpha)^p-s\alpha=s(\alpha^p-\alpha)=sg_1=g_2.
\]
Therefore, \(s\alpha\) and \(\beta\) are solutions of the equation \(x^p-x=g_2\). Thus \(\beta=s\alpha+t\) for some \(t\in\FF_p\), and hence \(\beta\in L^\sigma\).
That contradicts to that \(L/K\) is \(G\)-extension. Therefore, we have \(g_2\not\in\FF_pg_1\).
\end{proof}

We denote the \(G\)-extension field \(L\) of \(K\) corresponding \((g_1,g_2)\in J^{(2)}\) by \(L_{g_1,g_2}\).

Let \(V\) be the natural representation of \(G\) of rank two over \(\mathcal O_K\). We
denote \(\bsym v_V(L_{g_1,g_2})\) by \(\bsym v_V(g_1,g_2)\). Let \(\mathcal O_K[x_1,x_2]\) be the coordinate ring of \(V\). 

\begin{prop}\label{prop:isomTM}
The tuning module \(\Xi_{L/K}^V\) for a \(G\)-extension field \(L\) of \(K\) is isomorphic to
\[
\Theta_L:=\{m\in\mathcal O_L\mid 
\ m(\sigma-1)^2=0,\ m(\tau-1)=am(\sigma-1)\}
\]
as an \(\mathcal O_K\)-module by the following maps:
\begin{gather*}
\Xi_{L/K}^V\ni\varphi\mapsto\varphi(x_2)\in\Theta_L\\
\Theta_L\ni m\mapsto(m(\sigma-1))\varphi_1+m\varphi_2\in\Xi_{L/K}^V
\end{gather*}
where \(\varphi_1,\varphi_2\) are \(\mathcal O_K\)-linear homomorphisms from the linear part of \(\mathcal O_K[x_1,x_2]\) to \(O_L\) defined by
\[
\varphi_i(x_j)=\begin{cases}
1&(i=j)\\
0&(i\ne j)
\end{cases}
\]
\end{prop}
\begin{proof}
Let \(\varphi\in\Xi_{L/K}^V\) and let \(m=\varphi(x_2)\). Since \(\varphi\) is \(G\)-equivariant, we get
\begin{gather*}
\begin{split}
m(\sigma-1)&=\varphi(x_2(\sigma-1))\\
&=\varphi(x_1),
\end{split}\\
\begin{split}
m(\sigma-1)^2&=\varphi(x_1(\sigma-1))\\
&=\varphi(0)=0,
\end{split}
\end{gather*}
and
\begin{align*}
m(\tau-1)&=\varphi(x_2(\tau-1))\\
&=\varphi(ax_1)\\
&=am(\sigma-1).
\end{align*}
Thus \(m\in\Theta_L\).

On the other hand, for \(m\in\Theta_L\), we define \(\varphi=(m(\sigma-1))\varphi_1+m\varphi_2\). Since \(m(\sigma-1)^2=0\), we have
\begin{align*}
\varphi(x_1)\sigma&=m(\sigma-1)\sigma\\
&=m((\sigma-1)^2+(\sigma-1))\\
&=m(\sigma-1)\\
&=\varphi(x_1)=\varphi(x_1\sigma)
\end{align*}
and
\begin{align*}
\varphi(x_2)\sigma&=m\sigma\\
&=m(\sigma-1)+m\\
&=\varphi(x_1+x_2)=\varphi(x_2\sigma).
\end{align*}
Since \(m(\tau-1)=am(\sigma-1)\), we have
\begin{align*}
\varphi(x_1)\tau&=m(\sigma-1)\tau\\
&=m((\sigma-1)(\tau-1)+(\sigma-1))\\
&=am(\sigma-1)^2+m(\sigma-1)\\
&=m(\sigma-1)\\
&=\varphi(x_1)=\varphi(x_1\sigma)
\end{align*}
and
\begin{align*}
\varphi(x_2)\tau&=m\tau\\
&=m(\tau-1)+m\\
&=am(\sigma-1)+m\\
&=\varphi(ax_1+x_2)=\varphi(x_2\tau).
\end{align*}
These equations show that \(\varphi\) is \(G\)-equivariant because \(G\) is generated by \(\sigma, \tau\) and \(x_1,x_2\) forms \(\mathcal O_K\)-basis of the linear part of \(\mathcal O_K[x_1,x_2]\). Therefore, \(\varphi\in\Xi_{L/K}^V\).

As above, we have two \(\mathcal O_K\)-linear homomorphisms
\begin{gather*}
\Xi_{L/K}^V\ni\varphi\mapsto\varphi(x_2)\in\Theta_L,\\
\Theta_L\ni m\mapsto(m(\sigma-1))\varphi_1+m\varphi_2\in\Xi_{L/K}^V.
\end{gather*}
These maps are inverse of each other.
\end{proof}
\begin{thm}\label{thm:main}
For a subgroup \(G\subset\SL_2(k)\) generated by
\[\sigma=\begin{bmatrix}
1&1\\
0&1
\end{bmatrix},\ \tau=\begin{bmatrix}
1&a\\
0&1
\end{bmatrix}\quad(a\in k\backslash\FF_p),\]
we consider the \(G\)-representation \(V=\AA_{\mathcal O_K}^2\). For \((g_1,g_2)\in J^{(2)}\), we get
\[
\bsym v_{V}(g_1,g_2)=\ceil{-\frac{\min\{v_K(g_1),pv_K(f)\}}{p^2}}
\]
where \(f=a^pg_1+g_2\).
\end{thm}
\begin{proof}
We put \(L=L_{g_1,g_2}\) for \((g_1,g_2)\in J^{(2)}\). Let \(\alpha,\beta\in L\) be generators of \(L^\sigma, L^\tau\) over \(K\) satisfying \(\alpha\tau=\alpha+1,\ \beta\sigma=\beta+1,\ \alpha^p-\alpha=g_1\), and \(\beta^p-\beta=g_2\). The fixed field \(L^\sigma\) is naturally regarded as a cyclic representation of \(\angles{\tau}\) over \(K\). Since
\[
\alpha^i(\tau-1)=(\alpha+1)^i-\alpha^i=i\alpha^{i-1}+\sum_{l=0}^{i-2}\binom il\alpha^l
\]
for \(i\geq 1\), \(\alpha^i(\tau-1)^j\) is a polynomial of \(\alpha\) degree \(i-j\). In particular, \(\alpha^{p-1}(\tau-1)^{p-1}\ne 0\). Hence the representation matrix of \(\tau\) on \(L^\sigma\) has only one Jordan block. Hence, we can choose \(K\)-basis \((A_i)_{i=0}^{p-1}\) of \(L^\sigma\) which satisfies
\[A_i(\tau-1)=A_{i-1}\quad(1\leq i\leq p-1)\]
and \(A_1=\alpha\). Similarly, we can choose \(K\)-basis \((B_i)_{i=0}^{p-1}\) of \(L^\tau\) which satisfies
\[
B_i(\sigma-1)=B_{i-1}\quad(1\leq i\leq p-1)
\]
and \(B_1=\beta\). Then \((A_iB_j)_{i,j=0}^{p-1}\) is a \(K\)-basis of \(L\).

We compute a \(\mathcal O_K\)-basis of \(\Theta_L\), which defined in Proposition \ref{prop:isomTM}. Let \(m\in\Theta_L\). We can write \(m=\sum_{i=0}^{p-1}\sum_{j=0}^{p-1}c_{ij}A_iB_j\) with \(c_{ij}\in K\). Since \(A_i\in L^\sigma\), we have
\[(A_iB_j)(\sigma-1)=(A_i\sigma)(B_j\sigma)-A_iB_j=A_i(B_j(\sigma-1))=A_iB_{j-1}.\]
Similarly, we have
\[(A_iB_j)(\tau-1)=A_{i-1}B_j.\]
Then we get
\begin{align*}
m(\sigma-1)^2&=\sum_{i=0}^{p-1}\sum_{j=2}^{p-1}c_{ij}A_iB_{j-2}\\
&=\sum_{i=0}^{p-1}\sum_{j=0}^{p-3}c_{i,j+2}A_iB_j
\end{align*}
Since \(m\in\Theta_L\), \(m\) satisfies \(m(\sigma-1)^2=0\). Hence we get
\[
m=\sum_{i=0}^{p-1}(c_{i0}A_i+c_{i1}A_iB_1).
\]
Moreover \(m\) satisfies \(m(\tau-1)=am(\sigma-1)\). We have
\begin{gather*}
\begin{split}
m(\tau-1)&=\sum_{i=0}^{p-1}(c_{i0}(A_i(\tau-1))+c_{i1}(A_iB_1(\tau-1)))\\
&=\sum_{i=1}^{p-1}(c_{i0}A_{i-1}+c_{i1}A_{i-1}B_1)\\
&=\sum_{i=0}^{p-2}(c_{i+1,0}A_i+c_{i+1,1}A_iB_1),
\end{split}\\
\begin{split}
am(\sigma-1)&=a\sum_{i=0}^{p-1}(c_{i0}(A_i(\sigma-1))+c_{i1}(A_iB_1(\sigma-1)))\\
&=\sum_{i=0}^{p-1}ac_{i1}A_iB_0.\\
\end{split}\\
\end{gather*}
Combining these, we get
\[
ac_{i1}=c_{i+1,0}
\]
for \(0\leq i\leq p-2\) and
\[
c_{i1}=0
\]
for \(1\leq i\leq p-1\). These implies that
\[m=c_{00}+c_{i1}(a\alpha+\beta).\]
We put \(\gamma=a\alpha+\beta\) and \(s=v_L(\gamma)\). We have
\begin{align*}
N_{L/L^\sigma}(\gamma)&=\prod_{i=0}^{p-1}\gamma\sigma^i=\prod_{i=0}^{p-1}(\gamma+i)\\
&=\gamma^p-\gamma=(a^p\alpha^p+\beta^p)-(a\alpha+\beta)\\
&=a^p(\alpha+g_1)-a\alpha+(\beta^p-\beta)\\
&=(a^p-a)\alpha+a^pg_1+g_2=(a^p-a)\alpha+f.
\end{align*}
Thus
\[
s=-v_{L^\sigma}((a^p-a)\alpha+f).
\]
Now
\[
v_{L^\sigma}(\alpha)=v_K(N_{L^\sigma/K}(\alpha))=v_K(g_1),
\]
which is not divided by \(p\). Since
\[
v_{L^\sigma}(f)=pv_K(f)\in p\ZZ,
\]
we get \(v_{L^\sigma}((a^p-a)\alpha)\ne v_{L^\sigma}(f)\). Hence
\begin{align*}
s&=-v_{L^\sigma}((a^p-a)\alpha+f)\\
&=-\min\{v_{L^\sigma}((a^p-a)\alpha),v_{L^\sigma}(f)\}\\
&=-\min\{v_K(g_1),pv_K(f)\}.
\end{align*}
Note that \(s\not\in p^2\ZZ\). Therefore,
\[
m\in\mathcal O_L\Leftrightarrow c_{00}\in\mathcal O_K\text{ and }c_{i1}\in t^{\ceil{\frac s{p^2}}}\mathcal O_K.
\]
Thus \(1,t^{\ceil{\frac s{p^2}}}\gamma\) is an \(\mathcal O_K\)-basis of \(\Theta_L\). This corresponds to an \(\mathcal O_K\)-basis 
\[
\varphi_2,\ t^{\ceil{\frac s{p^2}}}\gamma\varphi_1+t^{\ceil{\frac s{p^2}}}\varphi_2
\]
of the tuning module \(\Xi_{L/K}^V\).

By formula \ref{form}, we get
\[
\bsym v_V(L)=\frac 1{p^2}v_L\left(\det\begin{bmatrix}
0&1\\
t^{\ceil{\frac s{p^2}}}&t^{\ceil{\frac s{p^2}}}\gamma
\end{bmatrix}\right)=\ceil{\frac s{p^2}}.
\]
As \(s=-\min\{v_K(g_1),pv_K(f)\}\), we get
\[\bsym v_V(g_1,g_2)=\ceil{-\frac{\min\{v_K(g_1),pv_K(f)\}}{p^2}}.\]
\end{proof}
\begin{coro}\label{cor:main}
We keep the notation of this section. Then the value of \(\bsym v_V\) at a \(G\)-extension \(L/K\) is not determined by the ramification filtration of \(G\) associated to \(L/K\).
\end{coro}
\begin{proof}
In the situation of Theorem \ref{thm:main}, we set \(g_1=t^{-(p^2-1)}\) and \(g_2=ct^{-(p^2-1)}+t^{-1}\) where \(c\in k-\FF_p\). By Theorem \ref{thm:main}, we get 
\[
\bsym v_V(g_1,g_2)=\begin{cases}
p&(c\ne -a^2)\\
1&(c=-a^2)
\end{cases}
\]
In particular, \(\bsym v_V(g_1,g_2)\) depends on the value of \(c\).

On the other hand, the upper ramification filtration is compatible with passing to a quotient group. Hence the upper ramification filtration is determined from the ones of all the intermediate fields of \(L/K\) by Proposition \ref{prop:upper}. Any intermediate field is determined from a subgroup of \(\FF_pg_1+\FF_pg_2\) by Artin-Schreier theory. An upper ramification filtration of an intermediate field is determined from the valuations of all elements of the corresponding subgroup of \(\FF_pg_1+\FF_pg_2\). As \(c\not\in\FF_p\), each nonzero element of \(\FF_pg_1+\FF_pg_2\) has valuation \(-(p^2-1)\). Therefore, the ramification filtration of \(L_{g_1,g_2}\) does not depend on \(c\).
\end{proof}

\end{document}